\documentclass[11pt,a4paper]{article}
\usepackage[utf8]{inputenc}
\usepackage{amsmath, amssymb, mathrsfs, amsthm, amsfonts}
\usepackage{latexsym}
\usepackage{a4wide}

\author{Hery Randriamaro\footnote{This research was supported by DAAD.}  \\  \footnotesize{Fachbereich Mathematik und Informatik}\\
\footnotesize{Philipps-Universit\"at Marburg}\\  \footnotesize{D-35032 Marburg} \\ \small{\texttt{herand@mathematik.uni-marburg.de}}}

\title{Spectral Properties of Descent Algebra Elements}

\begin{document}

\maketitle

\begin{abstract}
\noindent The descent algebra of finite Coxeter groups is studied by many famous mathematicians like Bergeron, Brown, Howlett, or Reutenauer.
Blessenohl, Hohlweg, and Schocker, for example, proved a symmetry property of the descent algebra, when it is linked to the representation theory of
its Coxeter group. The interest is particularly showed for the descent algebra of symmetric group. Thibon determined the eigenvalues and
their multiplicities of the action on the group algebra of symmetric group of the descent algebra element, which is the sum over all permutations
weighted by \begin{math}q^{\mathtt{maj}}\end{math}. And even the author diagonalized the matrix of the action of the descent algebra element, which
is the sum over all permutations weighted by the new introduced statistic \begin{math}\mathtt{des}_X\end{math}. In this article, we give a more
general result by determining the eigenvalues and their multiplicities of the action on the group algebra of finite Coxeter group of an element of its
descent algebra. 
\end{abstract}

\section{Introduction}

\noindent We keep the usual notations \begin{math}\mathbb{K}\end{math} for an algebraically closed field of characteristic \begin{math}0\end{math},
and \begin{math}(W,S)\end{math} for a finite Coxeter system, that is to say, \begin{math}W\end{math} is a finite group generated by the elements of
\begin{math}S\end{math} subject to the defining relations \begin{displaymath}(sr)^{m_{sr}}=e, \ \text{for all}\ s,r \in S\end{displaymath} where
\begin{math}e\end{math} is the neutral element of \begin{math}W\end{math}, the \begin{math}m_{sr}\end{math} are positive integers, and
\begin{math}m_{ss}=1\end{math} for all \begin{math}s \in S\end{math}.

\smallskip

\noindent Let \begin{math}J \subseteq S\end{math}. We naturally use the notations \begin{math}W_J\end{math} for the parabolic subgroup of
\begin{math}W\end{math} generated by the elements of \begin{math}J\end{math}, and \begin{math}c_J\end{math} for a Coxeter element of
\begin{math}W_J\end{math} which is a product of the elements of \begin{math}J\end{math} taken in some fixed order. We write
\begin{math}\overline{c_J}\end{math} for the conjugacy class of \begin{math}c_J\end{math}.

\smallskip

\noindent Let \begin{math}J \subseteq S\end{math}. We write \begin{math}\widetilde{J}\end{math} for the set of subsets
\begin{math}K \subseteq S\end{math} such that the parabolic subgroups \begin{math}W_J\end{math} and \begin{math}W_K\end{math} are conjugate. Let
\begin{math}\widetilde{J_1}, \dots, \widetilde{J_p}\end{math} be pairwise different such that
\begin{displaymath}\{\widetilde{J_i}\}_{i \in [p]} = \{\widetilde{J}\}_{J \subseteq S}.\end{displaymath}

\noindent Let \begin{math}J,K \subseteq S\end{math}. We write \begin{math}{}^JW^K\end{math} for the distinguished cross section for the double coset
space \begin{math}W_J \backslash W / W_K\end{math}. If \begin{math}J=\{\varnothing\}\end{math} resp. \begin{math}K=\{\varnothing\}\end{math}, we just
write \begin{math}W^K\end{math} resp. \begin{math}{}^JW\end{math}.

\smallskip

\noindent Let \begin{math}J \subseteq S\end{math}. We write \begin{displaymath}x_J := \sum_{w \in W^J} w.\end{displaymath}

\noindent Let \begin{math}\Xi_W := \{x_J \,|\,J \subseteq S\}\end{math}. We know that \begin{math}\mathbb{K}[\Xi_W]\end{math} is the descent algebra
of \begin{math}W\end{math} which is a subalgebra of the group algebra \begin{math}\mathbb{K}[W]\end{math} multiplicatively equipped with
\cite[Theorem~1]{So 1976}: \begin{displaymath}x_J x_K = \sum_{ L \subseteq K} a_{JKL} x_L\end{displaymath}
for \begin{math}J,K \subseteq S\end{math}, with
\begin{displaymath}a_{JKL}:=\big| \{ x \in {}^JW^K\ |\ x^{-1} W_J \ x \cap W_K = W_L \} \big|.\end{displaymath}

\noindent Let \begin{math}u = \sum_{w \in W} \lambda_w w \in \mathbb{K}[W]\end{math}. We write \begin{math}R_W(u) = (\lambda_{ww'^{-1}})_{w,\,w' \in W}\end{math} for
the matrix of the regular representation of the left-multiplication action of \begin{math}u\end{math} on \begin{math}\mathbb{K}[W]\end{math} relatively
to the standard basis \begin{math}\{w \,|\, w \in W\}\end{math}. The purpose of this article is to prove the following theorem:

\newtheorem{MapCla}{Theorem}[section] 
\begin{MapCla} \label{edesc}
Let \begin{math}d = \sum_{J \subseteq S} \lambda_J x_J \in \mathbb{K}[\Xi_W]\end{math}. Then the spectrum of \begin{math}R_W(d)\end{math} is
\begin{displaymath}Sp\big( R_W(d) \big) = \Big\{\varDelta_j = \sum_{i=1}^p  a_{J_i J_j J_j} \big( \sum_{K_i \in \widetilde{J_i}} \lambda_{K_i} \big)  \Big\}_{j \in [p]} \end{displaymath}
with corresponding multiplicities \begin{displaymath}\big\{ m_{\varDelta_j}= |\overline{c_{J_j}}| \big\}_{j \in [p]}.\end{displaymath}
\end{MapCla}

\noindent Let \begin{math}w \in W\end{math}. Recall that the set of left resp. right descent set of \begin{math}w\end{math} is
\begin{displaymath}\mathtt{DES_L}(w):= \{s \in S\,|\,\mathtt{l}(sw) < \mathtt{l}(w) \}\ \text{resp.}\
\mathtt{DES_R}(w):= \{s \in S\,|\,\mathtt{l}(ws)< \mathtt{l}(w) \},\end{displaymath}
\begin{math}\mathtt{l}\end{math} being the statistic which gives the length of the minimal expression of the elements of \begin{math}W\end{math}
in terms of elements of \begin{math}S\end{math}.

\noindent Let \begin{math}J \subseteq S\end{math}. We write \begin{math}D_J\end{math} for the set of elements of \begin{math}W\end{math} with right
descent set \begin{math}J\end{math}. We have the disjoint union
\begin{displaymath}W^{S \setminus K} = \bigcup_{J \subseteq K} D_J.\end{displaymath}
Thus setting \begin{displaymath}y_J := \sum_{w \in D_J} w,\end{displaymath}
we get \begin{displaymath}x_{S \setminus K} = \sum_{J \subseteq K} y_J,\end{displaymath}
and hence by M\"obius inversion formula \begin{displaymath}y_J = \sum_{K \subseteq J} (-1)^{|J \setminus K|} x_{S \setminus K}.\end{displaymath}

\noindent The descent algebra is a well studied object on the borderline of combinatorics and algebra (\cite{BB 1992}, \cite{BBHT 1992},
\cite{BHS 2005}). Especially the descent algebra relative to the Coxeter system of the symmetric group \begin{math}(A_n, S_{A_n}=\{t_1, \dots, t_n\})\end{math},
where \begin{math}t_i\end{math} is the transposition \begin{math}(i \ i+1)\end{math} (\cite{G 1995}, \cite{K 1997}). Thibon determined the eigenvalues
and their multiplicities of the action of the element \cite[Theorem~56]{K}
\begin{displaymath}\sum_{J \subseteq S_{A_n}} q^{\mathtt{Maj}(J)} \, y_J \in \mathbb{R}(q)[\Xi_{A_n}],\end{displaymath}
where \begin{displaymath}\mathtt{Maj}(J):= \sum_{j \in \{i \in [n]\,|\,t_i \in J\}}j.\end{displaymath}
In \cite{Sc}, Schocker worked on the descent algebra of the symmetric group. Using hyperplanes arrangement and random walk properties (\cite{Br 2000},
\cite{Br 2003}), Brown determined the condition of diagonalizability of the action of element of the descent algebra of finite Coxeter group, and gave a
remarkable approach of the eigenvalues and corresponding multiplicities of its regular representation. In \cite{Ra}, the regular representation of the element
\begin{displaymath}\sum_{J \subseteq S_{A_n}}\mathtt{Des}_X(J) \, y_J \in \mathbb{R}(X_1, \dots, X_n)[\Xi_{A_n}],\end{displaymath}
where \begin{displaymath}\mathtt{Des}_X(J):= \sum_{j \in \{i \in [n]\,|\,t_i \in J\}}X_j,\end{displaymath} was diagonalized

\smallskip

\noindent This paper is organized as follows. We begin with the calculation of the formula for the coefficients \begin{math}a_{J_i J_j J_j}\end{math}.
Then, we determine the eigenvalues and their corresponding multiplicities of the regular representation of an element of the descent algebra of a finite Coxeter
group. In the appendix part, we treat the Coxeter group \begin{math}F_4\end{math} as complete example, and we give a counterexample of the formula
in \cite[Theorem~6.5]{BBHT 1992}.

\section{Special Coefficients of the Descent Algebra}

\noindent In this section, we determine a formula for the values of \begin{math}a_{JKK}\end{math}. Through a slight modification, our formula corrects
a mistake from the formula proposed in \cite[Theorem~6.5]{BBHT 1992}.

\noindent Let \begin{math}J \subseteq S\end{math}. We write \begin{math}N_J \end{math} for the subgroup \cite[Corollary~3]{Ho 1980}
\begin{displaymath}N_J:= \{w \in W \,|\, w^{-1} W_J w = W_J \} \cap W^J. \end{displaymath}

\noindent Let \begin{math}w \in W\end{math} and \begin{math}U,V\end{math} be subgroups of \begin{math}W\end{math}. We write
\begin{displaymath} {}^U w^V := \{x \in UwV\ |\ \mathtt{l}(x) \leq \mathtt{l}(uxv),\, \forall u \in U,\, \forall v \in V\}\end{displaymath} 
for the set of minimal double coset representatives of \begin{math}w\end{math} relative to \begin{math}U\end{math} and \begin{math}V\end{math}.
If \begin{math}U = \{e\} \end{math} resp. \begin{math}V = \{e\} \end{math}, we just write \begin{math}w^V \end{math} resp. 
\begin{math}{}^U w \end{math}. For the case of parabolic subgroups, we just write \begin{math}{}^{W_J} w^{W_K} = {}^J w^K\end{math} with
\begin{math}J,K \subseteq S\end{math}. This lemma can be read off from \cite[1.Introduction]{So 1976}.

\newtheorem{corep}[MapCla]{Lemma}
\begin{corep}
Let \begin{math}w \in W\end{math} and \begin{math}J, K \subseteq S\end{math}. Then the set \begin{math}{}^J w^K\end{math} contains a unique element.
In this case, we consider \begin{math}{}^J w^K\end{math} no more as a subset of \begin{math}W\end{math} but as an element of \begin{math}W\end{math}.
\end{corep}

\noindent Let \begin{math}K \subseteq S\end{math} and \begin{math}K' \in \widetilde{K}\end{math}. We write
\begin{displaymath}C_{K' K}:= \{w \in W \ |\ w^{-1} W_{K'}\ w  = W_K \}.\end{displaymath}

\newtheorem{corepc}[MapCla]{Lemma}
\begin{corepc}
Let \begin{math}c \in  C_{K' K}\end{math}, and \begin{math}c_{K' K} \in  {}^{{}^{N_{K'} W_{K'}}} c^{{}^{N_{K} W_{K}}}\end{math}.
Then \begin{math}c_{K' K} \in  C_{K' K}\end{math}.
\end{corepc}

\begin{proof} If \begin{math}c = n'\, k'\, c_{K' K}\, n\, k\end{math} with \begin{math}n' k'  \in  N_{K'} W_{K'} \end{math} and \begin{math}nk  \in  N_K W_K \end{math}, then 
\begin{displaymath} \left.  \begin{array}{ccc}
(n'\, k'\, c_{K' K}\, n\, k)^{-1}W_{K'}\ n'\, k'\, c_{K' K}\,  n\, k  & = & W_K,   \\
( c_{K' K}\,  n\, k)^{-1}W_{K'}\ c_{K' K}\,  n\, k &  = & W_K  ,\\ 
c_{K' K}^{-1} W_{K'}\  c_{K' K} & = &  n\, k\, W_K\, (n\, k)^{-1}, \\
c_{K' K}^{-1}  W_{K'}\  c_{K' K} & = &  W_K.  \end{array}  \right.  \end{displaymath}
\end{proof}

\noindent Let \begin{math}E \subseteq  W\end{math} and \begin{math}U,V\end{math} be subgroups of \begin{math}W\end{math}. We write 
\begin{displaymath}{}^UE^V := \bigcup_{w \in E} {}^U w^V.\end{displaymath}
For the case of parabolic subgroups, we just write \begin{math}{}^{W_J} E^{W_K} = {}^J E^K\end{math}, with \begin{math}J,K \subseteq S\end{math}.

\newtheorem{farany}[MapCla]{Lemma}
\begin{farany}
Let \begin{math}K \subseteq S\end{math}, \begin{math}K' \in \widetilde{K}\end{math}, and
\begin{math}c_{K' K} \in  {}^{{}^{N_{K'} W_{K'}}} C_{K' K}{}^{{}^{N_{K} W_{K}}}\end{math}. Then 
\begin{displaymath} {}^{K'}(c_{K' K} N_K)^K =  c_{K'K} N_K\ \text{and}\ {}^{K'}(N_{K'}\, c_{K' K})^K = N_{K'}\, c_{K' K}.\end{displaymath}
\end{farany}

\begin{proof} It is clear that \begin{math}(c_{K' K} N_K)^K =  c_{K' K} N_K \end{math}. Effectively, since \begin{math}c_{K' K} \end{math} is a left coset representative of the subgroup 
\begin{math} N_K W_K \end{math}, and the elements of \begin{math} N_K\end{math} are left coset representatives of \begin{math} W_K\end{math}, then
the elements of \begin{math}c_{K' K}  N_K\end{math} are left coset representatives of \begin{math} W_K\end{math}.\\
\noindent We just then have to prove that \begin{math} {}^{K'}(c_{K' K} N_K) =  c_{K' K} N_K \end{math}.
Let \begin{math} c_{K' K}\,n \in c_{K' K} N_K\end{math} and let us suppose that \begin{math}c_{K' K}\,n = k' b\end{math}, where \begin{math}k' \in  W_{K'} \end{math}
and \begin{math}b \in {}^{K'} W\end{math}. Then \begin{math}\mathtt{l}(c_{K' K}\,n) =  \mathtt{l}(k' b) =  
\mathtt{l}(k' ) + \mathtt{l}(b) \end{math} i.e. \begin{displaymath}\mathtt{l}(c_{K' K}\,n) \geq  \mathtt{l}(b).\end{displaymath}
On the other hand, we have \begin{math}(k')^{-1}c_{K' K}\,n = c_{K' K}\, k_1 n = c_{K' K}\, n k_2 = b \end{math} with \begin{math}k_1, k_2 \in  W_K \end{math}. 
Then \begin{math}\mathtt{l}(c_{K' K}\,n k_2) = \mathtt{l}(c_{K' K}\,n) + \mathtt{l}(k_2)  =  \mathtt{l}(b) \end{math}, i.e. 
\begin{displaymath}\mathtt{l}(c_{K' K}\,n)  \leq  \mathtt{l}(b).\end{displaymath}
The only possibility is then \begin{math}k_2 = k_1 = k'  =  e\end{math}, so we get the result.\\
\noindent The proof for \begin{math}{}^{K'}(N_{K'}\, c_{K' K})^K = N_{K'}\, c_{K' K}\end{math} is analogous.
\end{proof}

\newtheorem{najkklem}[MapCla]{Lemma}
\begin{najkklem}
Let \begin{math}K \subseteq S\end{math}, \begin{math}K' \in \widetilde{K}\end{math}, and
\begin{math}c_{K' K} \in  {}^{{}^{N_{K'} W_{K'}}} C_{K' K}{}^{{}^{N_{K} W_{K}}}\end{math}. Then 
\begin{displaymath}\{w \in {}^{K'} W^K\ |\ w^{-1} W_{K'}\ w  = W_K\} = c_{K' K} N_K = N_{K'}\, c_{K' K}.\end{displaymath}
\end{najkklem}

\begin{proof} It is clear that \begin{math}\{ w \in {}^K W^K\ |\ w^{-1} W_K\ w  = W_K \} = N_K \end{math} for all \begin{math}K \subseteq S\end{math}. 

\begin{itemize}
\item The map \begin{math}\phi:\{ w \in {}^{K'} W^{K'}\ |\ w^{-1} W_{K'}\ w = W_{K'} \} \rightarrow  \{ w \in {}^{K'} W^K\ |\ w^{-1} W_{K'}\ w  = W_K \} \end{math}
\begin{displaymath}n \mapsto n\,c_{K'K}  \end{displaymath}  
is clearly injective.

\item The map \begin{math}\phi':\{ w \in {}^{K'} W^K\ |\ w^{-1}  W_{K'}\ w  = W_K \} \rightarrow  \{ w \in {}^{K'} W^{K'}\ |\ w^{-1} W_{K'}\ w  = W_{K'} \} \end{math},
\begin{displaymath}x \mapsto {}^{K'}(c_{K' K}\,  x^{-1})^{K'}\end{displaymath}   
is injective. Effectively, \begin{math}\phi'(x)=\phi'(y)\end{math} means  \begin{math}c_{K' K}\,x^{-1}= u_1 c_{K' K}\,y^{-1} u_2\end{math} with \begin{math}u_1, u_2 \in  W_{K'} \end{math}.
Then \begin{math}c_{K' K}\, x^{-1}=c_{K' K} v y^{-1} u_2\end{math} and \begin{math}x^{-1}= v y^{-1} u_2\end{math} with \begin{math}v \in W_K\end{math}.
The only possibility is \begin{math}v = u_2 = e\end{math}. 
\end{itemize}
\noindent Then we deduce that \begin{math}\phi\end{math} is bijective and \begin{math}\{w \in {}^{K'} W^K\ |\ w^{-1} W_{K'}\ w = W_K \} = N_{K'}\, c_{K' K}\end{math}.\\
\noindent The proof is analogous for \begin{math}\{ w \in {}^{K'} W^K\ |\ w^{-1} W_{K'}\ w  = W_K \}  = c_{K' K} N_K \end{math}.
\end{proof}

\noindent Let \begin{math}J, K \subseteq S\end{math}, and, for all \begin{math}K' \in \widetilde{K}\end{math}, let us fixe an element
\begin{math}c_{K' K} \in  {}^{{}^{N_{K'} W_{K'}}} C_{K' K}{}^{{}^{N_{K} W_{K}}}\end{math}. We write
\begin{displaymath}H_{JK} := \{K' \in \widetilde{H} \,|\, c_{K' K} = {}^J c_{K' K} \}.\end{displaymath}
We can now give the formula to determine \begin{math}a_{JKK}\end{math}.

\newtheorem{new}[MapCla]{Theorem}
\begin{new} \label{new}
Let \begin{math}J, K \subseteq S\end{math}. We have 
\begin{displaymath}a_{JKK} = \sum_{K' \in H_{JK} \cap 2^J} \frac{|N_K|}{|W_J \cap N_{ K'}|}.  \end{displaymath}
\end{new}

\begin{proof} Recall that \begin{displaymath}a_{JKK} := \big| \{ w \in {}^J W^K\ |\ W_J \cap w W_K w^{-1} = w W_K w^{-1}\} \big|,\end{displaymath}
with \begin{math}w W_K w^{-1} =  W_{K'} \end{math} and \begin{math}K' \in  \widetilde{K} \cap 2^J  \end{math}.\\
We have \begin{math}w \in  c_{K' K}  N_K =  N_{K'}\, c_{K' K} \end{math}. But we must also have \begin{math}w = {}^J w\end{math}. That means:
\begin{itemize}
\item On the one hand, we must have \begin{math}c_{K' K} = {}^J c_{K' K}\end{math}. Otherwise \begin{math} {}^J(c_{K' K}N_K)  \cap  C_{K' K} = \varnothing\end{math}. 
\item On the other hand, if \begin{math}c_{K' K} = {}^J c_{K' K}\end{math}, then \begin{math} w \in {}^J(N_{K'}\, c_{K' K}) = ({}^J N_{K'})\,c_{K' K} \end{math}.
\end{itemize}
Since \begin{displaymath}\big|{}^J N_{K'}\big| = \frac{|N_{K'}|}{|W_J \cap N_{K'}|}, \end{displaymath} it follows that
\begin{align*}
a_{JKK} =\ & \sum_{K' \in H_{JK} \cap 2^J} \frac{|N_{K'}|}{|W_J \cap N_{ K'}|}\\ 
=\ & \sum_{K' \in H_{JK} \cap 2^J} \frac{|N_K|}{|W_J \cap N_{ K'}|}.
\end{align*}
since \begin{math}W_K\end{math} and \begin{math}W_{K'}\end{math} are conjugate.
\end{proof}

\noindent Let \begin{math}J' \in \widetilde{J}\end{math} and \begin{math}K' \in \widetilde{K}\end{math}. From Theorem \ref{new}, we deduce that 
\begin{displaymath}a_{JKK} = a_{J'K'K'}.\end{displaymath} This result can also be found in \cite[Theorem~6.2]{BBHT 1992}.

\section{Eigenvalues and Multiplicities}

\noindent We are now able to determine the eigenvalues and their corresponding multiplicities.\\
\noindent Let \begin{math}d = \sum_{J \subseteq S} \lambda_J x_J \in \mathbb{K}[\Xi_W]\end{math}. We write \begin{math}v_{\Xi_W}(d)\end{math}
for the column vector of \begin{math}d\end{math} relative to the basis \begin{math}\Xi_W\end{math}, and \begin{math}M_{\Xi_W}(d)\end{math} for the
matrix of the left-multiplication action of \begin{math}d\end{math} on \begin{math}\mathbb{K}[\Xi_W]\end{math} relative the basis
\begin{math}\Xi_W\end{math} i.e. \begin{displaymath}v_{\Xi_W}(d)=(\lambda_K)_{K \subseteq S}\ \ \text{and}\ \
M_{\Xi_W}(d) = \big( \sum_{J \subseteq S} \lambda_J a_{JKL} \big)_{K, L \subseteq S}.\end{displaymath}

\noindent Let \begin{math}n \geq 2\end{math} and \begin{math}(d_i)_{i \in [n]} \in \mathbb{K}[\Xi_W]^n\end{math}. We have
\begin{displaymath}v_{\Xi_W}(\prod_{i \in [n]}^{\rightarrow} d_i) =
\big( \prod_{i \in [n-1]}^{\rightarrow}M_{\Xi_W}(d_i) \big) \cdot v_{\Xi_W}(d_n).\end{displaymath}

\noindent Recall that the noncommutative multiplication is defined in the following way:
\begin{displaymath}\prod_{i \in [n]}^{\rightarrow} \lambda_i = \lambda_1 \lambda_2 \dots \lambda_n.\end{displaymath}

\noindent We write \begin{math}M_{\Xi_W}(d)_{|\bullet,K}\end{math} for the column of \begin{math}M_{\Xi_W}(d)\end{math} corresponding to the basis
vector \begin{math}x_K\end{math}, i.e.
\begin{displaymath}M_{\Xi_W}(d)_{|\bullet,K} = \big( \sum_{J \subseteq S} \lambda_J a_{JKL} \big)_{L \subseteq S} =v_{\Xi_W}(d \cdot x_K).\end{displaymath}

\newtheorem{relation}[MapCla]{Lemma}
\begin{relation} \label{relation}
Let \begin{math}d \in \mathbb{K}[\Xi_W]\end{math}. Then \begin{math}R_W(d)\end{math} and \begin{math}M_{\Xi_W}(d)\end{math} have the same spectrum.
\end{relation}

\begin{proof} It is clear that \begin{math}Sp\big(M_{\Xi_W}(d)\big) \subseteq Sp\big(R_W(d)\big)\end{math}.\\
We write \begin{math}0_{2^{|S|}}\end{math} for the matrix with entry \begin{math}0\end{math} of \begin{math}\mathbb{K}^{2^{|S|} \times 2^{|S|}}\end{math},
and \begin{math}I_{2^{|S|}}\end{math} for the identity matrix of \begin{math}\mathbb{K}^{2^{|S|} \times 2^{|S|}}\end{math}.
Let \begin{math}\sum_{i=0}^{2^{|S|}}\mu_i t^i\end{math} be the characteristic polynomial of \begin{math}M_{\Xi_W}(d)\end{math} in the variable
\begin{math}t\end{math} with \begin{math}(\mu_i)_{i \in \{0\} \cup [2^{|S|}]} \in \mathbb{K}^{2^{|S|}+1}\end{math}. We have
\begin{math}\sum_{i=0}^{2^{|S|}}\mu_i M_{\Xi_W}(d)^i = 0_{2^{|S|}}\end{math}, especially
\begin{displaymath}\mu_0 I_{2^{|S|}|\bullet,S}  + \sum_{i=1}^{2^{|S|}}\mu_i M_{\Xi_W}^{i-1}(d) \cdot M_{\Xi_W}(d)_{|\bullet,S} = 0_{2^{|S|}|\bullet,S}.\end{displaymath}
Since \begin{math}I_{2^{|S|}|\bullet,S}= v_{\Xi_W}(e)\end{math}, and \begin{math}M_{\Xi_W}(d)_{|\bullet,S} = v_{\Xi_W}(d)\end{math}, then
\begin{displaymath}\mu_0 v_{\Xi_W}(e) + \sum_{i=1}^{2^{|S|}}\mu_i \, M_{\Xi_W}^{i-1}(d) \cdot v_{\Xi_W}(d) = v_{\Xi_W}(0).\end{displaymath}
This means \begin{math}\mu_0 e + \sum_{i=1}^{2^{|S|}}\mu_i d^i=0\end{math} and \begin{math}Sp\big(R_W(d)\big) \subseteq Sp\big(M_{\Xi_W}(d)\big)\end{math}.
\end{proof}

\noindent For the rest of the section, we need a total order \begin{math}\succ\end{math} on the subsets of \begin{math}S = \{s_i\}_{i \in \big[|S|\big]}\end{math}
which was introduced by F. and N. Bergeron \cite{BB 1992}:
We define \begin{math}\min\,J :=\min \{i \in \big[|S|\big]\,|\,s_i \in J\} \end{math}, and assume that \begin{math}\min \varnothing = |S|+1\end{math}.
Let \begin{math}J, K \subseteq S\end{math} such that \begin{math}J \neq K\end{math}.
\begin{itemize}
\item If \begin{math}\min\,J > \min\,K \end{math} then \begin{math}J \succ K\end{math}.
\item Otherwise \begin{math}J \succ K\end{math} if and only if \begin{math}J \setminus \{s_{\min\,J}\} \succ K \setminus \{s_{\min\,K}  \}. \end{math}
\end{itemize}

\noindent We have already seen the definition of the set \begin{math}\{\widetilde{J_i}\}_{i \in [p]}\end{math} in the introduction.\\ 
Let \begin{math}K_i\end{math} be the element of \begin{math}\widetilde{J_i}\end{math} such that \begin{math}L_i \succ K_i\end{math}
for all \begin{math}L_i \in \widetilde{J_i} \setminus \{K_i\} \end{math}. We order the sets
\begin{math}\{\widetilde{J_i}\}_{i \in [p]}\end{math} such that \begin{math}K_i \succ K_j\end{math} if \begin{math}i<j\end{math}.

\newtheorem{eigenwerte}[MapCla]{Proposition}
\begin{eigenwerte} \label{eigen}
Let \begin{math}d = \sum_{J \subseteq S} \lambda_J x_J \in \mathbb{K}[\Xi_W]\end{math}. Then the spectrum of the matrix \begin{math}R_W(d)\end{math} is
\begin{displaymath}Sp\big(R_W(d)\big) = \Big\{\sum_{i=1}^p  a_{K_i K_j K_j} \big( \sum_{L_i \in \widetilde{J_i}} \lambda_{L_i}\big) \Big\}_{j \in [p]}.\end{displaymath}
\end{eigenwerte}

\begin{proof} We order the basis \begin{math}(x_J)_{J \subseteq S}\end{math} of \begin{math}\mathbb{K}[\Xi_W]\end{math} according
to the total order \begin{math}\succ\end{math} of F. and N. Bergeron that means we get a new ordered basis \begin{math}(x_{L_i})_{i \in [2^{|S|}]}\end{math}
such that \begin{math}L_i \succ L_j\end{math} if \begin{math}i < j\end{math}. We assume that the rows and columns of the matrix
\begin{math}M_{\Xi_W}(d)\end{math} are ordered in increasing order by the order \begin{math}\succ\end{math}. We get
\begin{math}d \cdot x_{L_j} \in \big<\{x_{L_i}\}_{i \in [j]}\big>  \end{math}. Then the matrix of \begin{math}d\end{math} on the basis
\begin{math}(x_{L_i})_{i \in [2^{|S|}]}\end{math} is an upper triangular matrix. The characteristic polynomial of this matrix in the variable
\begin{math}t\end{math} is \begin{displaymath}\prod_{K \subseteq S} \big( (\sum_{J \subseteq S} \lambda_J  a_{JKK}) -t \big),\end{displaymath} 
and \begin{displaymath}Sp(M_{\Xi_W}(d)) = \big\{ \sum_{J \subseteq S} \lambda_J a_{JKK} \big\}_{K \subseteq S}.   \end{displaymath}
Since \begin{math}a_{JKK} = a_{J' K'K'} \end{math} for all \begin{math}J, J' \in \widetilde{J_i}\end{math}, and all
\begin{math}K, K' \in \widetilde{J_j}\end{math}, we get the result.
\end{proof}

\newtheorem{multipli}[MapCla]{Proposition}
\begin{multipli} \label{multipli}
Let \begin{math}d = \sum_{J \subseteq S} \lambda_J x_J \in \mathbb{K}[\Xi_W]\end{math}, and
\begin{math}\varDelta_j = \sum_{i=1}^p a_{K_i K_j K_j} \big( \sum_{L_i \in \widetilde{J_i}} \lambda_{L_i} \big)\end{math}. Then the multiplicity of the
eigenvalue \begin{math}\varDelta_j\end{math} of \begin{math}R_W(d)\end{math} is \begin{displaymath}m_{\varDelta_j} = |\overline{c_{J_j}}|.\end{displaymath}
\end{multipli}

\begin{proof} Let \begin{math}A = (a_{K_i K_j K_j})_{i,j \in [p]}\end{math}, \begin{math}m = (m_{\varDelta_j})_{j \in [p]}\end{math},
\begin{math}u = (|W|)_{j \in [p]}\end{math}, and \begin{math}c = (|\overline{c_{J_j}}|)_{j \in [p]}\end{math}. At the end of the sixth section of
\cite{BBHT 1992}, it is proved that \begin{math}A^{-1} u = c\end{math}.\\ 
\noindent Let \begin{math}tr\end{math} be the trace map of square matrix. We have \begin{math}tr\big( R_W(d) \big) = |W| \sum_{J \subseteq S} \lambda_J\end{math}.
Then, \begin{math}\sum_{i=1}^j a_{K_j K_i K_i}\, m_{\varDelta_i}\, \lambda_{L_j} = |W| \lambda_{L_j} \end{math} for \begin{math} L_j \in \widetilde{J_j}\end{math}
i.e. \begin{displaymath}\sum_{i=1}^j a_{K_j K_i K_i}\, m_{\varDelta_i} = |W|.\end{displaymath}
In matrix form, we get \begin{math}A m = u\end{math}. Thus \begin{math}A^{-1} u = m\end{math}.
\end{proof}

\noindent We note that with the matrix relation \begin{math}A m = u\end{math}, we can also get the cardinalities of the conjugacy classes of
\begin{math}W\end{math}.

\appendix

\section{Example of the Symmetry Group of 24-cell}

\noindent Recall that the Coxeter system of the symmetry group of 24-call with cardinality \begin{math}1152\end{math} is
\begin{math}(F_4, S_{F_4} = \{s_i \}_{i \in [4]})\end{math}, and its Coxeter graph is
\begin{displaymath}s_1  \longleftrightarrow  s_2 \stackrel{4}{\longleftrightarrow}  s_3  \longleftrightarrow s_4\end{displaymath}
Using Theorem \ref{new} and the values of \begin{math}|N_K|\end{math} in \cite[page~74]{Ho 1980}, we get the following values of
\begin{math}a_{JKK}\end{math} for the case of \begin{math}F_4\end{math}:

{\footnotesize
{\renewcommand{\arraystretch}{1.5}
\renewcommand{\tabcolsep}{0.1cm}
\begin{center}
\begin{tabular}{|c|c|c|c|c|c|c|c|c|c|c|c|c|}
\hline
   & \begin{math} \varnothing \end{math}   & \begin{math} \{s_1\} \end{math}   &  \begin{math} \{s_4\} \end{math}   & \begin{math} \{s_1, s_2 \} \end{math}   &  
\begin{math} \{s_2, s_3 \} \end{math}   &  \begin{math}\{s_3, s_4 \}  \end{math}  &   \begin{math}\{s_1, s_4 \}  \end{math}  &  \begin{math}\{s_1, s_2, s_3 \} \end{math}  & 
\begin{math}\{s_2, s_3, s_4 \} \end{math}  &    \begin{math}\{s_1, s_3, s_4 \} \end{math}  & \begin{math}\{s_1, s_2, s_4 \} \end{math}  &    \begin{math} S_{F_4} \end{math}   \\  \cline{1-13}
\begin{math} \varnothing \end{math}    &  1152  &  0 & 0  & 0 & 0 & 0  &  0 &  0 &  0 & 0  & 0 & 0 \\  \cline{1-13}
\begin{math} \{s_1\} \end{math}   &  576  &  48   & 0  & 0   &  0  & 0  &  0   & 0   &  0 & 0 & 0  & 0 \\  \cline{1-13}
\begin{math} \{s_4 \} \end{math}    &  576  & 0  & 48   & 0   & 0  &  0 & 0   & 0    &  0 & 0  & 0  & 0 \\  \cline{1-13}
\begin{math} \{s_1, s_2 \} \end{math}    &  192  & 48  & 0  &  12  & 0  & 0  & 0  &  0  &  0  & 0  & 0   &  0\\  \cline{1-13}
\begin{math} \{s_2, s_3 \} \end{math}    &  144  & 24  & 24  &  0  & 8  & 0  & 0  &  0  &  0  & 0  & 0  & 0 \\  \cline{1-13}
\begin{math} \{s_3, s_4 \} \end{math}    &  192  & 0  & 48  &  0  & 0  & 12  & 0  &  0  &  0  & 0  & 0  & 0  \\  \cline{1-13}
\begin{math} \{s_1, s_4 \} \end{math}    & 288   & 24  & 24   &  0 &  0  &  0 &  4   & 0   & 0& 0 & 0 & 0 \\  \cline{1-13}
\begin{math} \{s_1, s_2, s_3 \} \end{math}    &  24  & 24  & 6   & 12   & 4  &  0 &   0 &  2  & 0 & 0  & 0 & 0 \\  \cline{1-13}
\begin{math} \{s_2, s_3, s_4 \} \end{math}    &  24  & 6  & 24   & 0   & 4  &  12 &  0  &  0   & 2  &  0 &  0 &  0 \\  \cline{1-13}
\begin{math} \{s_1, s_3, s_4 \} \end{math}    &  96  & 8  & 24   & 0   & 0  &  6 &  4  &   0  & 0 & 2 & 0 & 0 \\  \cline{1-13}
\begin{math} \{s_1, s_2, s_4 \} \end{math}    &  96  & 24  & 8   & 6   & 0  &  0 &  4 &  0   & 0 & 0  & 2 & 0  \\  \cline{1-13}
\begin{math} S_{F_4} \end{math}    &  1   & 1  & 1  & 1 & 1   &  1 &  1   &  1   &  1 & 1  & 1  & 1 \\  \cline{1-13}
\hline
\end{tabular}
\end{center}}
}

\smallskip

\noindent We consider the element \begin{displaymath}d = \sum_{J \subseteq S_{F_4}} \lambda_J x_J \in \mathbb{K}[\Xi_{F_4}].\end{displaymath}  
The eigenvalues of \begin{math}R_{F_4}(u)\end{math} are
\small{
\begin{displaymath} \left.
\begin{array}{lll}
\varDelta_1 & = &  1152 \lambda_{\varnothing} + 576 ( \lambda_{\{s_1\}} + \lambda_{\{s_2\}} + \lambda_{\{s_3\}} +  \lambda_{\{s_4\}} ) + 192 (\lambda_{\{s_1, s_2 \}} + \lambda_{\{s_3, s_4 \}}) + 144 \lambda_{\{s_2, s_3 \}}\\ 
&& + 288 (\lambda_{\{s_1, s_3 \}} +  \lambda_{\{s_1, s_4\}} + \lambda_{\{s_2, s_4  \}}) + 24 ( \lambda_{\{s_1, s_2, s_3 \}} +     \lambda_{\{s_2, s_3, s_4 \}} ) + 96  ( \lambda_{\{s_1, s_3, s_4 \}} +  
\lambda_{\{s_1, s_2, s_4 \}} ) \\ \smallskip
&& + \lambda_{S_{F_4}}   \\  
  
\varDelta_2 & = &  48 (\lambda_{\{s_1\}} + \lambda_{\{s_2\}} +  \lambda_{\{s_1, s_2 \}})  + 24 ( \lambda_{\{s_2, s_3 \}} + \lambda_{\{s_1, s_3 \}} +  \lambda_{\{s_1, s_4\}} + \lambda_{\{s_2, s_4  \}} +  \lambda_{\{s_1, s_2, s_3 \}}+ 
\lambda_{\{s_1, s_2, s_4 \}} )  \\  \smallskip
&& + 6 \lambda_{\{s_2, s_3, s_4 \}} + 8 \lambda_{\{s_1, s_3, s_4 \}}  + \lambda_{S_{F_4}}  \\

\varDelta_3 & = &  48 (\lambda_{\{s_3\}} +  \lambda_{\{s_4\}} +  \lambda_{\{s_3, s_4 \}}) + 24 ( \lambda_{\{s_2, s_3 \}} +  \lambda_{\{s_1, s_3 \}} +  \lambda_{\{s_1, s_4\}} + \lambda_{\{s_2, s_4\}} +  \lambda_{\{s_2, s_3, s_4 \}} + 
\lambda_{\{s_1, s_3, s_4 \}}  )   \\  \smallskip
&&+ 6 \lambda_{\{s_1, s_2, s_3 \}} + 8  \lambda_{\{s_1, s_2, s_4 \}} + \lambda_{S_{F_4}} \\ \smallskip

\varDelta_4 & = &  12 \lambda_{\{s_1, s_2 \}} + 12 \lambda_{\{s_1, s_2, s_3 \}} + 6 \lambda_{\{s_1, s_2, s_4 \}}+ \lambda_{S_{F_4}}  \\  \smallskip

\varDelta_5 & = &  8  \lambda_{\{s_2, s_3 \}} + 4 (\lambda_{\{s_1, s_2, s_3 \}}+ \lambda_{\{s_2, s_3, s_4 \}}) + \lambda_{S_{F_4}} \\  \smallskip

\varDelta_6 & = &  12( \lambda_{\{s_3, s_4 \}} +  \lambda_{\{s_2, s_3, s_4 \}}) +  6 \lambda_{\{s_1, s_3, s_4 \}} + \lambda_{S_{F_4}}   \\  \smallskip

\varDelta_7 & = &  4  (\lambda_{\{s_1, s_3 \}} +  \lambda_{\{s_1, s_4\}} + \lambda_{\{s_2, s_4  \}}) + \lambda_{\{s_1, s_3, s_4 \}} +  \lambda_{\{s_1, s_2, s_4 \}} )+ \lambda_{S_{F_4}}  \\  \smallskip

\varDelta_8 & = & 2 \lambda_{\{s_1, s_2, s_3 \}}+ \lambda_{S_{F_4}}  \\   \smallskip

\varDelta_9 & = &  2 \lambda_{\{s_2, s_3, s_4 \}}+ \lambda_{S_{F_4}}  \\   \smallskip

\varDelta_{10} & = & 2 \lambda_{\{s_1, s_3, s_4 \}} + \lambda_{S_{F_4}}  \\  \smallskip

\varDelta_{11} & = & 2  \lambda_{\{s_1, s_2, s_4 \}} + \lambda_{S_{F_4}}  \\  \smallskip

\varDelta_{12} & = &  \lambda_{S_{F_4}}  
\end{array} \right.
\end{displaymath}}

\normalsize{\noindent with corresponding multiplicities}

\begin{displaymath}
\begin{array}{ccccc}
m_{\varDelta_1} & = &   1   & = &  |\overline{e}|  \\
m_{\varDelta_2} & = &   12   & = &  |\overline{s_1}|  \\
m_{\varDelta_3} & = &   12  & = &  |\overline{s_4}|  \\
m_{\varDelta_4} & = &   32   & = &  |\overline{s_1 s_2}|  \\
m_{\varDelta_5} & = &   54  &  = &  |\overline{s_2 s_3}|  \\
m_{\varDelta_6} & = &   32   &  = &  |\overline{s_3 s_4}|  \\
m_{\varDelta_7} & = &  72    &  = &  |\overline{s_1 s_4}|  \\
m_{\varDelta_8} & = &   84   &  = &  |\overline{s_1 s_2 s_3}|  \\
m_{\varDelta_9} & = &   84   &  = &  |\overline{s_2 s_3 s_4}|  \\
m_{\varDelta_{10}} & = &   96   &  = &  |\overline{s_1 s_3 s_4}|  \\
m_{\varDelta_{11}} & = &   96  &  = &  |\overline{s_1 s_2 s_4}|  \\
m_{\varDelta_{12}} & = &   577  &  = &  |\overline{s_1 s_2 s_3 s_4}|  
\end{array}
\end{displaymath}

\section{Counterexample on the Special Coefficients}

\noindent The following formula is proposed in \cite[Theorem~6.5]{BBHT 1992}: For \begin{math}J, K \subseteq S\end{math},
\begin{displaymath}a_{JKK} = \sum_{K' \in \widetilde{K} \cap 2^J} \frac{|N_K|}{|W_J \cap N_{K'}|}.\end{displaymath}

\noindent Recall that the Coxeter system of the symmetry group of the dodecahedron with cardinality \begin{math}120\end{math} is
\begin{math}(H_3, S_{H_3} = \{s_i \}_{i \in [3]})\end{math}, and its Coxeter graph is
\begin{displaymath}s_1 \stackrel{5}{\longleftrightarrow}  s_2  \longleftrightarrow s_3 \end{displaymath}
We have the values of \begin{math}|N_K|\end{math} for \begin{math}H_3\end{math} in \cite[page~79]{Ho 1980}. If we use the formula in
\cite[Theorem~6.5]{BBHT 1992}, we get the following values of \begin{math}a_{JKK}\end{math} for \begin{math}H_3\end{math}:

{\renewcommand{\arraystretch}{1.5}
\renewcommand{\tabcolsep}{0.1cm}
\begin{center}
\begin{tabular}{|c|c|c|c|c|c|c|}
\hline
   & \begin{math} \varnothing \end{math}   & \begin{math} \{s_1\} \end{math}   & \begin{math} \{s_1, s_2 \} \end{math}   & \begin{math}\{ s_2, s_3 \} \end{math}
  & \begin{math}\{s_1,  s_3 \}  \end{math}  &   \begin{math} S_{H_3} \end{math}   \\  \cline{1-7}
\begin{math} \varnothing \end{math}    &  120  &  0 & 0  & 0 & 0 & 0    \\  \cline{1-7}
\begin{math} \{s_1\} \end{math}   &  60  &  4   & 0  & 0   &  0  & 0 \\  \cline{1-7}
\begin{math} \{s_1, s_2 \} \end{math}    &  12  & 8  & 2  &  0  & 0  & 0 \\  \cline{1-7}
\begin{math} \{ s_2, s_3 \} \end{math}    &  20  & 8  & 0   & 2   & 0  &  0 \\  \cline{1-7}
\begin{math} \{s_1, s_3 \} \end{math}    & 30   & 4  & 0   &  0 &  2  &  0    \\  \cline{1-7}
\begin{math} S_{H_3} \end{math}    &  1   & 1  & 1  & 1 & 1   &  1  \\  \cline{1-7}
\hline
\end{tabular}
\end{center}}

\noindent Let \begin{math}\mathsf{A} = \left( \begin{array}{cccccc}
120 & 0 & 0 & 0 & 0 & 0 \\
60 & 4 & 0 & 0 & 0 & 0 \\
12 & 8 & 2 & 0 & 0 & 0 \\
20 & 8 & 0 & 2 & 0 & 0 \\
30 & 4 & 0 & 0 & 2 & 0 \\
1 & 1 & 1 & 1 & 1 & 1 \end{array} \right)    \end{math}. We know from Proposition \ref{multipli} that
\begin{math}\mathsf{A}^{-1} \left( \begin{array}{c} 120 \\ 120 \\ 120 \\ 120 \\ 120 \\ 120  \end{array} \right) \end{math} gives the cardinalities of the
conjugacy classes of \begin{math}H_3\end{math}. However, we get 
\begin{displaymath}\mathsf{A}^{-1} \left( \begin{array}{c} 120 \\ 120 \\ 120 \\ 120 \\ 120 \\ 120  \end{array} \right) =
\left( \begin{array}{c} 1 \\ 15 \\ -6 \\ -10 \\ 15 \\ 105  \end{array} \right)  \end{displaymath} which is absurd.

\noindent But if we use Theorem \ref{new}, the values of \begin{math}a_{JKK}\end{math} calculated for \begin{math}H_3\end{math} are:

{\renewcommand{\arraystretch}{1.5}
\renewcommand{\tabcolsep}{0.1cm}
\begin{center}
\begin{tabular}{|c|c|c|c|c|c|c|}
\hline
   & \begin{math} \varnothing \end{math}   & \begin{math} \{s_1\} \end{math}   & \begin{math} \{s_1, s_2 \} \end{math}   & \begin{math}\{ s_2, s_3 \} \end{math}  & \begin{math}\{s_1,  s_3 \}  \end{math}  &   \begin{math} S_{H_3} \end{math}   \\  \cline{1-7}
\begin{math} \varnothing \end{math}    &  120  &  0 & 0  & 0 & 0 & 0    \\  \cline{1-7}
\begin{math} \{s_1\} \end{math}   &  60  &  4   & 0  & 0   &  0  & 0 \\  \cline{1-7}
\begin{math} \{s_1, s_2 \} \end{math}    &  12  & 4  & 2  &  0  & 0  & 0 \\  \cline{1-7}
\begin{math} \{ s_2, s_3 \} \end{math}    &  20  & 4  & 0   & 2   & 0  &  0 \\  \cline{1-7}
\begin{math} \{s_1, s_3 \} \end{math}    & 30   & 4  & 0   &  0 &  2  &  0    \\  \cline{1-7}
\begin{math} S_{H_3} \end{math}    &  1   & 1  & 1  & 1 & 1   &  1  \\  \cline{1-7}
\hline
\end{tabular}
\end{center}}

\noindent Let \begin{math}\mathsf{B} = \left( \begin{array}{cccccc}
120 & 0 & 0 & 0 & 0 & 0 \\
60 & 4 & 0 & 0 & 0 & 0 \\
12 & 4 & 2 & 0 & 0 & 0 \\
20 & 4 & 0 & 2 & 0 & 0 \\
30 & 4 & 0 & 0 & 2 & 0 \\
1 & 1 & 1 & 1 & 1 & 1 \end{array} \right)    \end{math}. Then we get the following cardinalities of the conjugacy classes of \begin{math}H_3\end{math}
\begin{displaymath}\mathsf{B}^{-1} \left( \begin{array}{c} 120 \\ 120 \\ 120 \\ 120 \\ 120 \\ 120  \end{array} \right) = 
 \left( \begin{array}{c} 1 \\ 15 \\ 24 \\ 20 \\ 15 \\ 45  \end{array} \right)  \end{displaymath}
which are the correct values.

\small{
\bibliographystyle{abbrvnat}
}

\end{document}